\documentclass[14pt]{amsart}
\address{\newline{\normalsize Max Planck Institute for Mathematics, Vivatsgasse 7, 53111 Bonn, Germany}
\newline{\it E-mail address}: karzhema@mpim-bonn.mpg.de}
\usepackage{amscd,amsthm,amsmath,amssymb}
\usepackage[dvips]{graphicx}
\usepackage[matrix,arrow]{xy}

\makeatletter\@addtoreset{equation}{section}\makeatother

\makeatletter\@addtoreset{subsection}{equation}\makeatother

\newtheorem{theorem}[equation]{Theorem}
\newtheorem{proposition}[equation]{Proposition}
\newtheorem{lemma}[equation]{Lemma}
\newtheorem{corollary}[equation]{Corollary}
\newtheorem{question}[equation]{Question}
\newtheorem{theorem-definition}[equation]{Theorem-definition}

\theoremstyle{definition}
\newtheorem{example}[equation]{Example}
\newtheorem*{notation}{Notation}

\theoremstyle{remark}
\newtheorem{remark}[equation]{Remark}

\newcommand{\dia}{\frak{D}}

\newcommand{\cg}{\textbf{Cr}}

\newcommand{\li}{{\bf\frak{G}}}
\newcommand{\slo}{{\bf\mathcal{E}}}

\newcommand{\com}{\mathbb{C}}
\newcommand{\f}{\mathbb{F}}
\newcommand{\p}{\mathbb{P}}
\newcommand{\re}{\mathbb{R}}
\newcommand{\ra}{\mathbb{Q}}
\newcommand{\cel}{\mathbb{Z}}
\newcommand{\na}{\mathbb{N}}
\newcommand{\au}{\textbf{UAut}}
\newcommand{\aut}{\textbf{Aut}}
\newcommand{\su}{\mathrm{Supp}}

\newcommand{\kr}{\mathrm{Ker}}
\newcommand{\vol}{\mathrm{Vol}}
\newcommand{\di}{\mathrm{dist}}

\thanks{{\it MS 2010 classification}: 14E07, 14R10, 20F69}

\thanks{{\it Key words}: Cremona group, affine automorphism, non\,-\,simplicity}

\begin{document}

\title{Combinatorics of affine birational maps}

\author{Ilya Karzhemanov}

\begin{abstract}
The main object of study in the present paper is the group $\au_n$
of \emph{unimodular automorphisms} of $\com^n$. Taking $\au_n$ as
a working example, our intention was to develop an approach (or
rather an edifice) which allows one to prove, for instance, the
non\,-\,simplicity of $\au_n$ for all $n \geq 3$. More systematic
and, perhaps, general exposition will appear elsewhere.
\end{abstract}

\sloppy

\maketitle

\bigskip

\section{Introduction}
\label{section:introduction}

The impetus for the present paper was the article \cite{bergman}
in which the study of combinatorics of certain birational
automorphisms of ${\bf k}^n$ was applied to answer a
group\,-\,theoretic question.\footnote{~Throughout the paper, if
not stated otherwise, ${\bf k}$ is an algebraically closed field
of characteristic $0$ and ${\bf k}^n$ is the $n$\,-\,dimensional
affine space.} More specifically, given $f$ from the \emph{Cremona
group} $\cg_n$ of birational automorphisms of ${\bf k}^n$, the
combinatorics of $f$ we have in mind is encoded (somehow) in a set
of lattice points or rather a polytope, which comes for free with
each $f\in\cg_n$. These discrete gadgets are constructed, as used
to be common now, by fixing a (non\,-\,archimedean) valuation on
the field of rational functions of ${\bf k}^n$ and applying this
valuation to the components of various maps $f\in\cg_n$. Or,
heuristically, one ``\,brings the action of $f$ on ${\bf k}^n$ to
infinity\," (see \cite{dan}, \cite{gromov-pol-growth},
\cite{gromov-2}, \cite{kapranov}, \cite{grisha}, \cite{max-yura}
for related matters).

We would like to apply the preceding point of view to study
polynomial automorphisms of ${\bf k}^n$. Recall that the group
$\aut_n$ of such automorphisms carries a structure of an
infinite\,-\,dimensional algebraic group (see \cite{shaf}). Then,
as an algebraic group, $\aut_n$ is generated by the group of
affine linear automorphisms of ${\bf k}^n$ and by the group of
triangular automorphisms of ${\bf k}^n$ (see \cite[Theorem
4]{shaf}). The group $\aut_n$ is also non\,-\,simple because of
the Jacobi map $\det : \aut_n \longrightarrow {\bf k}^*$. On the
other hand, the kernel $\au_n := \kr(\det)$ is simple as an
algebraic group (see \cite[Theorem 5]{shaf}), but is not that as
an abstract group for $n = 2$ (see \cite{danilov}). The aim of the
present paper is to extend the latter result to the case of
arbitrary $n \geq 3$:

\begin{theorem}
\label{theorem:main} The group $\au_n$ is non\,-\,simple (as an
abstract group) for all $n \geq 3$.
\end{theorem}

To prove Theorem~\ref{theorem:main}, we introduce a subgroup $G
\subset \cg_n$ that ``\,looks like\," a subgroup in $SL_n(\cel)$
when brought to infinity --- according to what we have said at the
beginning (see Section~\ref{section:1-0-n} for the construction of
$G$). Though the presence of $G$ might be interesting and
important on its own (see for example
Proposition~\ref{theorem:g-is-not-embeddable}), we focus on one of
its subgroups, namely $\li_n \subset G$, instead (see
Corollary~\ref{theorem:rho-hom}). For technical reasons, one
should also consider an ``\,enlargement\,'' of $\li_n$, which we
denote $\widetilde{\li_n}$ (see
Section~\ref{section:theorem-group-e-n}). One of the crucial
features of $\widetilde{\li_n}$ is provided by
Proposition~\ref{theorem:sub-group-surj}. Up to this end all
considerations employ only elementary algebra/combinatorics of
polynomials in $n$ variables.

Proposition~\ref{theorem:sub-group-surj} is enough to prove
Theorem~\ref{theorem:main} provided that $\widetilde{\li_n}$
contains \emph{sufficiently many} normal subgroups. The latter
turns out to be the case after we introduce a subset $G_e$ of
generators of $\widetilde{\li_n}$ in
Section~\ref{section:intermedia-group}. More specifically, $G_e$
is stable under the conjugation and inversion in
$\widetilde{\li_n}$, and generates a normal subgroup $N \ne
\widetilde{\li_n}, \{1\}$. Here we apply (seem to be standard) an
argument \`a la geometric group theory, based on the notion of
\emph{quasi\,-\,isometry},\footnote{~The results of
Section~\ref{section:intermedia-group} were motivated by (and are
a group\,-\,theoretic counterpart of) the \emph{Splitting theorem}
for compact Lorentz manifolds with isometric
$SL_2(\mathbb{R})$\,-\,action (cf. \cite[\S 4]{ambra-gromov} and
Question~\ref{theorem:stupid} below).} together with simple
(asymptotic) properties of birational automorphisms of ${\bf k}^n$
(see Section~\ref{section:proof-of-theorem}).

In conclusion, let us say that despite a recent progress in
understanding the structure of $\cg_n$, presumably for $n \leq 3$
(see \cite{can-lam}, \cite{prok}, \cite{serre-1}), this group
still remains a mysterious creature and far more questions about
it are currently out of reach (see for example \cite{serre-1} for
a modest account of some of these). We hope that
Theorem~\ref{theorem:main} and the methods used to prove it will
shed the light on some part of the Cremona group (see
Section~\ref{section:theorem-proofs} for further discussion).
\\

\begin{notation}

Throughout the paper we use the following notation and
conventions:

\begin{itemize}

\item $\p^n$ is the projective space with coordinates $\left[X_0\,:\,\cdots\,:\,X_n\right]$.
We denote by $S:={\bf
k}\left[X_0,\ldots,X_n\right]_{\mathrm{hom}}$ the semigroup of
homogeneous polynomials in the ring ${\bf
k}\left[X_0,\ldots,X_n\right]$.

\smallskip

\item We fix the lattice $\cel^{n + 1}$ with the basis dual to
$\{X_0,\ldots,X_n\}$. We also fix the sublattice
$\cel^n\subset\cel^{n + 1}$ corresponding to $\{X_1,\ldots,X_n\}$.
Both $\cel^{n + 1}$ and $\cel^n$ are equipped with the standard
lexicographical order for which $X_0 \geq X_1 \geq\ldots\geq X_n$.

\smallskip

\item We set $X := (X_1,\ldots,X_n)$, $X^I := X_1^{i_1}
\ldots X_n^{i_n}$ for $I\in\cel^n$, $I := (i_1,\ldots,i_n)$. ${\bf
M}_n(R)$ denotes the set of all $(n \times n)$\,-\,matrices $M$
with entries $M_{i,j}\in R$, $1\leq i,j\leq n$, in a ring $R$.

\smallskip

\item Given $h\in{\bf
k}\left[X_0,\ldots,X_n\right]$ we set $d_h$ to be the degree of
$h$ in $X_0$. We denote by $\su~h$ the support of $h$ (i.\,e.
$\su~h$ is the collection of all monomials that appear in $h$ with
non\,-\,zero coefficients). We will identify $\su~h\subset \cel^{n
+ 1}$ with the dual set of monomials and denote by $I_h$ the
maximal vector among those $I\in\cel^n$ with $(d_h, I)\in\su~h$.
We also put $\left<h\right> := (d_h, I_h)$ (thus $\left<h\right>$
is the monomial $X_0^{d_h}X^{I_h}$).

\smallskip

\item Every $f\in\cg_n$ (and, more generally, every rational self\,-\,map of $\p^n$) is represented by an $(n + 1)$\,-\,tuple
$[f_0\,:\,\cdots\,:\,f_n]$ of (not necessarily coprime)
polynomials $f_0,\ldots,f_n\in {\bf
k}\left[X_0,\ldots,X_n\right]_{\mathrm{hom}}$. In particular, if
all $f_i$ are coprime, then $f$ is uniquely determined by
$[f_0\,:\,\cdots\,:\,f_n]$.

\smallskip

\item $f \circ g$ (or $fg$, or $f \cdot g$) denotes
the composition $f(g)$ of two rational self\,-\,maps of $\p^n$.

\smallskip

\item For a group $G$ and any $a_1, a_2, b \in G$, we put $a_i^b :=
ba_ib^{-1} $, $C_{a_i} := \{a_i^b\}_{b \in G}$ (the conjugacy
class of $a_i$), and write $a_1 \sim a_2$ if $a_1\in C_{a_2}$. $N
\vartriangleleft G$ signifies that $N$ is a normal subgroup in $G$
such that $N \ne G, \{1\}$ ($1\in G$ is the unit element).

\smallskip

\item We denote by $\f_2$ the free group in two generators ($\f_2$
always comes with the word metric with respect to a fixed set of
generators). We will also use standard notions and facts from the
geometric group theory (see e.\,g. \cite{gromov-2}). For instance,
given two metric spaces $X$ and $Y$, $X \sim_{\text{q.-i.}} Y$ (or
$X$ is q.\,-\,i. to $Y$) signifies that $X$ is quasi\,-\,isometric
to $Y$.

\end{itemize}

\end{notation}

\bigskip

\section{Preliminaries}
\label{section:1-0-n}

\refstepcounter{equation}
\subsection{The set-up}

We take ${\bf k} = \com$ for simplicity. Consider $f\in\cg_n$
given by some $f_0,\ldots,f_n\in S$ (we assume that $n \ge 3$ in
what follows). Suppose that
\begin{eqnarray}
\label{eq-1-1-0-n} f_0 = \alpha_0X_0^{d_f}X^{I_{f_0}} + \sum_{k
\geq 1}X_0^{d_f - k}F_{k}(X),\\
\nonumber f_j = \alpha_{j, 0}X_0^{d_f-1}X^{I_{f_j}} + \sum_{k \geq
1}X_0^{d_f - 1 - k}F_{j, k}(X)
\end{eqnarray}
for all $j \geq 1$ and some $d_f\in\na$, where $\alpha_0,
\alpha_{j,0} \in\com^*$, $F_k(X), F_{j, k}(X)\in
\com\left[X_1,\ldots,X_n\right]_{\mathrm{hom}}=S$ (cf. the {\bf
Notation} above). Note that the condition $d_{f_0} - d_{f_j} = 1$
(for all $j \geq 1$) is satisfied by every $(n + 1)$-tuple
$(f^*_0, \ldots, f^*_n)$ such that $hf^*_i = f_i$ for all $i$ and
any (fixed) $h \in S$ (with $f_i$ replaced by $f^*_i$ in $d_{f_i}$
for all $i$). In particular, we may take $f_i$ to be coprime,
$0\leq i\leq n$, so that \eqref{eq-1-1-0-n} is a property of the
map $f$.

Let us also assume that $f^{-1}$ satisfies \eqref{eq-1-1-0-n} and
denote by $G$ the set of all such $f$. Then clearly $G\ne\{1\}$:

\begin{example}
\label{example:ex-1-1-0-n} $G$ contains the following groups:

\begin{itemize}

\item the group $\dia_n := (\com^*)^n$ of diagonal automorphisms of
$\p^n$;

\smallskip

\item the subgroup in $\aut_n$ of those
$f$ which preserve the origin in $\com^n$ and have Jacobi matrix
equal $1$;

\smallskip

\item for each $M \in SL_n(\cel)$, with the $j$\,-\,th column
$I_j$ contained in the hyperplane $\displaystyle\sum_{i = 1}^nX_i
= 1$ for all $j$, the birational transformation
$[1\,:\,X_1\,:\,\cdots\,:\,X_n] \mapsto
[1\,:\,X^{I_1}\,:\,\cdots\,:\,X^{I_n}]$ (we identify $X_j$ with
\emph{affine} $X_j/X_0$) also satisfies \eqref{eq-1-1-0-n}. Note
that all such $M$ form a group isomorphic to the subgroup
$SL'_{n}(\cel)$ of those elements in $SL_n(\cel)$ that fix the
vector $(1,\ldots,1)$ (see {\ref{subsection:free-aut}} below for
an explicit example of two $a_1,a_2\in SL'_{n}(\cel)\subset G$).

\end{itemize}
Less trivial examples are provided by the groups $\li_n\subset G$
and $\left<\slo_n\right>\subseteq\li_n$ below.
\end{example}

Example~\ref{example:ex-1-1-0-n} justifies the existence of the

\refstepcounter{equation}
\subsection{Group structure on $G$}

Put $h(\left<f\right>) :=
h(\left<f_0\right>,\ldots,\left<f_n\right>) =
h\big(1,\displaystyle\frac{\left<f_1\right>}{\left<f_0\right>},\ldots,\displaystyle\frac{\left<f_n\right>}{\left<f_0\right>}\big)\left<f_0\right>^{\deg(h)}$
for every $h \in S$ and $f \in G$ as above. Let also $M_f$ be the
$(n \times n)$\,-\,matrix whose $j$\,-\,th column equals $I_{f_j}
- I_{f_0},1 \leq j \leq n$.

Suppose that $(d_h, I)\in\su~h$, $I\in\cel^n$, yields $I = I_h$.
In this setting we get the following:

\begin{lemma}
\label{theorem:cruc-lm-1} The equality
$$
\left<h(\left<f\right>)\right> = \big(\deg(h)(d_f - 1) +
d_h,~M_fI_h + \deg(h)I_{f_0}\big)
$$
holds.
\end{lemma}

\begin{proof}
Indeed, since $f\in G$, from \eqref{eq-1-1-0-n} and above
definitions we get (by direct substitution)
$$
\sigma(\left<f\right>)=\big(\deg(h)(d_f - 1) + d_{\sigma},~M_fI +
\deg(h)I_{f_0}\big) < \big(\deg(h)(d_f - 1) + d_h,~M_fI_h +
\deg(h)I_{f_0}\big) = \left<h(\left<f\right>)\right>
$$
for all
$\sigma:=X_0^{d_{\sigma}}X^I\in\su~h\setminus{\{\left<h\right>\}}$.
\end{proof}

Put $h(f) := h(f_0,\ldots,f_n) =
h\big(1,\displaystyle\frac{f_1}{f_0},\ldots,\displaystyle\frac{f_n}{f_0}\big)f_0^{\deg(h)}$.
Note that $\left<h(\left<f\right>)\right>\in\su~h(f)$. Then from
Lemma~\ref{theorem:cruc-lm-1} we obtain
\begin{equation}
\label{h-f} \left<h(f)\right> = \big(\deg(h)(d_f - 1) +
d_h,~M_fI_h + \deg(h)I_{f_0}\big).
\end{equation}

This leads to the anticipated

\begin{proposition}
\label{theorem:e-n-is-a-group} $G$ is a subgroup in $\cg_n$.
\end{proposition}

\begin{proof}
Take $f\in G$ as above. Consider also $g\in G$ given by some
$g_0,\ldots,g_n\in S$. Then from \eqref{h-f} (for $h = g_0,
\ldots, g_n$) we obtain that $g \circ f$ is of the form
\eqref{eq-1-1-0-n}. Recall also that $f^{-1}, g^{-1}\in G$ by
definition. Thus we get $(g \circ f)^{-1} = f^{-1} \circ g^{-1}\in
G$, which proves the assertion.
\end{proof}

\refstepcounter{equation}
\subsection{Homomorphism $\rho$}
\label{subsection:val-v}

Consider the map $v : S \longrightarrow \cel_{\scriptscriptstyle
\geq 0}^n$ defined as follows:
\begin{eqnarray}
\nonumber v : h \mapsto \left<h\right> = \big(d_h, I_h\big)
\mapsto I_h
\end{eqnarray}
for all $h \in S$. Then $v$ is a $\cel_{\scriptscriptstyle \geq
0}^n$\,-\,valuation on $S$. Furthermore, $v$ (obviously) extends
to a $\cel^n$\,-\,valuation on
$\com\big(X_1/X_0,\ldots,X_n/X_0\big)$, providing a particular
case of valuations considered in \cite{andrei-1}, \cite{andrei-2}.
This determines a map $$\rho : G \longrightarrow {\bf M}_n(\cel),\
f \mapsto \rho(f) := M_f$$ for all $f \in G$, with
$v(f_i/f_0)=I_{f_i}-I_{f_0}$, the $i$\,-\,th column of $M_f$,
$1\le i\le n$.

From Proposition~\ref{theorem:e-n-is-a-group} we get the
following:

\begin{corollary}
\label{theorem:rho-hom} $\rho(G) \subset GL_n(\cel)$ and $\rho$ is
a group homomorphism. In particular, for $\li_n := \kr(\rho)$ the
group $\mathrm{Out}\,\li_n$ of outer automorphisms of $\li_n$
contains $\f_2$.
\end{corollary}

\begin{proof}
Let us use the notation from the proof of
Proposition~\ref{theorem:e-n-is-a-group}. Recall that $d_{g_0(f)}
- d_{g_j(f)} = 1$ for all $j \geq 1$ and $I_{g_i(f)} = M_fI_{g_i}
+ \deg(g_0)I_{f_0}$ for all $i \geq 0$ (see \eqref{h-f}). This
implies that $\rho(g \circ f) = M_{g \circ f} = M_fM_g =
\rho(g)\rho(f)$. Note also that $\rho$ splits over
$SL'_{n}(\cel)\subset G$ and $\rho(G) = SL'_{n}(\cel)$ by
construction (see \eqref{eq-1-1-0-n} and
Example~\ref{example:ex-1-1-0-n}). Thus we get $G = \li_n \rtimes
SL'_{n}(\cel)$ and a homomorphism
$\f_2\longrightarrow\mathrm{Out}\,\li_n$ (cf.
Lemma~\ref{theorem:a-1-a-2-f-2} below). Let us show that the
latter is injective.

Consider an arbitrary $f\in\dia_n\subset G$ (cf.
Example~\ref{example:ex-1-1-0-n}). We may assume that $f$
coincides with the map
$$
[X_0\,:\,X_1\,:\,\cdots\,:\,X_n]\mapsto
[X_0\,:\,\lambda_1X_1\,:\,\cdots\,:\,\lambda_nX_n]$$ for some
fixed $\lambda_i\in\com^*$. Now take any $a\in\f_2\subseteq
SL'_{n}(\cel)$. Then $a(f)$ ($:=afa^{-1}$ in $G$) also belongs to
$\dia_n$ and is obtained from $f$ by replacing every
$\lambda_i,1\le i\le n$, by the product
$\displaystyle\prod_{j=1}^n\lambda_j^{k_{j,i}}$ for some
$k_{j,i}\in\cel$ such that $\displaystyle\sum_j k_{j,i}=1$. In
particular, one may always choose $f$ (for $a\ne 1$) to be such
that $\displaystyle\prod_{j=1}^n\lambda_j^{k_{j,i}}\ne \lambda_i$
for at least one $i$, so that $a(f)\ne f$ in this case.

On the other hand, if $a(f) = f^g$ for some $g\in\li_n$, then it
follows from Lemma~\ref{theorem:cruc-lm-1} that $f^g = f$ (cf.
\eqref{eq-1-1-0-m} below for a ``\,typical\," shape of $g$). This
together with $a(f)\ne f$ shows that $\f_2$ injects into
$\mathrm{Out}\,\li_n$.
\end{proof}

\bigskip

\section{The group $\left<\slo_n\right>$}
\label{section:theorem-group-e-n}

We retain the notation of Section~\ref{section:1-0-n}. Consider a
rational map $\Lambda : \p^n \dashrightarrow \p^n$ defined as
follows:
\begin{eqnarray}
\label{eq-1-1-0-m} X_0 \mapsto X_0^{d} + X_0X_1^{d - 1} =: \Lambda_0,\\
\nonumber X_1 \mapsto X_0^{d-1}X_1 +  X_1^d =: \Lambda_1,\\
\nonumber X_j \mapsto \alpha_{j}X_0^{d-1}X_j + \Lambda^*_{j}(X) =:
\Lambda_j
\end{eqnarray}
for all $j \geq 2$ and some $d\in\na$, $\alpha_j \in\com$,
$\Lambda^*_{j}(X)\in
\com\left[X_0,X_1,\ldots,X_n\right]_{\mathrm{hom}}$,
$d_{\Lambda^*_{j}} \le d-2$. Let us additionally assume that the
map
$$\Lambda^*\big\vert_{X_1 = 0}:\ [X_0\,:\,X_2\,:\,\cdots\,:\,X_n]
\mapsto [X_0^d\,:\,\alpha_{2}X_0^{d-1}X_2 +
\Lambda^*_{2}(X)\,:\,\cdots:\,\alpha_{n}X_0^{d-1}X_n +
\Lambda^*_{n}(X)] \mod X_1$$ is a birational automorphism of
$\p^{n - 1}$ which coincides with a polynomial automorphism from
$\au_{n - 1}$ on $\com^{n - 1} = \p^{n - 1} \cap (X_0 \ne 0)$.
Denote by $\slo_n$ the set of all such $\Lambda$ contained in
$\li_n$.

\begin{lemma}
\label{theorem:sub-group-surj-1} We have $\slo_n \ne \{1\}$. More
precisely, to every element in $\au_{n - 1}$, which preserves the
origin $0 \in \com^{n-1}$ and has diagonal Jacobi matrix, there
corresponds an element in $\slo_n$, similarly as
$\Lambda\in\slo_n$ above corresponds to $\Lambda^*\in\au_{n - 1}$.
\end{lemma}

\begin{proof}
One may identify $\Lambda$ with the map
$$
(X_1,\,\ldots,\,X_n)\mapsto\big(\Lambda_1/\Lambda_0 =
X_1,\,\ldots,\,\Lambda_n/\Lambda_0\big)
$$
on the affine chart $\com^{n} = \p^n\cap(X_0\ne 0)$. In short we
have $\Lambda:\, X \mapsto (X_1,\,(1 + X_1^{d-1})^{-1} \cdot
\Lambda^*)$. It is then plain that $\Lambda^{-1}\in\cg_n$ exists
and is of the form \eqref{eq-1-1-0-m}.
\end{proof}

Let $Aff_n(\com) \subset PGL_n(\com)$ be the subgroup of affine
transformations which preserve the hyperplane $\Pi := (X_1 = 0)
\subset \p^n$. Each $g \in Aff_n(\com)$ satisfies $g^*X_0 =
\lambda(g) X_0$ for some $\lambda(g) \in \com^*$ (similarly for
$X_1$) and the group $Aff_n(\com)$ is characterized by this
property.

Set $\widetilde{\li_n} \subset \cg_n$ to be the homomorphic image
of the amalgamated product $Aff_n(\com)*_{\,\dia_n}\li_n$. Recall
that $\f_2$ acts on $\li_n \supset \dia_n$ as earlier and $\dia_n$
is invariant under this action (cf. the proof of
Corollary~\ref{theorem:rho-hom}). Then, using the Iwasawa
decomposition for $GL_n(\com)$,\footnote{~Note that $\f_2$ acts
trivially on the compact torus $(S^1)^n \subset \dia_n$ because
$M(M^{-1}v) = v$ for any $M \in SL'_n(\cel)$ and $v \in
\mathbb{R}^n$.} one extends the action $\f_2 \circlearrowright
\dia_n$ to the whole group $Aff_n(\com)$ and hence to
$Aff_n(\com)*_{\,\dia_n}\li_n$ as well. Furthermore, since the
kernel of the natural homomorphism $Aff_n(\com)*_{\,\dia_n}\li_n
\to \widetilde{\li_n}$ is obviously $\f_2$\,-\,invariant by
construction, we get an $\f_2$\,-\,action on $\widetilde{\li_n}$
inducing the initial one on $\li_n$. Then the following holds:

\begin{lemma}
\label{theorem:tilde-g-aut} $\f_2 \subseteq \mathrm{Out}\,\li_n$
acts on $\widetilde{\li_n}$ by the outer automorphisms as well.
\end{lemma}

\begin{proof}
Suppose that $a(f) = f^g$ for all $a \in \f_2$, $f \in \dia_n$ and
some fixed $g = g_1 \cdot g_2\in\widetilde{\li_n}$, written in one
of its normal forms with $g_1 \in Aff_n(\com)$ and $g_2 \in
\li_n$. Then we obtain $f^{g_2} = (a(f))^{g_1^{-1}} \in
Aff_n(\com)$. This is only possible when $f^{g_2} = f$. Hence
necessarily $g_1 = 1$ and $a(f) = f$.
\end{proof}

Let $\left<\slo_n\right> \subseteq \widetilde{\li_n}$ be the
subgroup generated by $\slo_n$ and $Aff_n(\com)$.

\begin{proposition}
\label{theorem:sub-group-surj} There exists a surjective
homomorphism $\xi : \left<\slo_n\right> \twoheadrightarrow \au_{n
- 1}$.
\end{proposition}

\begin{proof}
$\xi$ is defined via restriction to the locus $\Pi \cap (X_0 \ne
0)$. Its surjectivity is immediate by
Lemma~\ref{theorem:sub-group-surj-1}.
\end{proof}

\begin{remark}
\label{remark:askold} Let $\Lambda\in\slo_n$ and $L \subseteq
\com(X_1,\ldots,X_n)$ be the linear subspace spanned by the
rational functions $1,\Lambda_0/\Lambda_1, \Lambda_2/\Lambda_1,
\ldots, \Lambda_n/\Lambda_1$. Then, as $\dim L = n+1$, the
\emph{(self) intersection index}
$[\underbrace{L,\ldots,L}_{\text{n + 1 times}}]$ (see e.\,g.
\cite{askold}) is equal to the degree of $\Lambda$, which is $1$
(cf. Lemma~\ref{theorem:sub-group-surj-1}). Let also $v$ be the
valuation as in {\ref{subsection:val-v}}. It follows from
\eqref{eq-1-1-0-m} that $\left\{v(\Lambda_1/\Lambda_0), \ldots,
v(\Lambda_n/\Lambda_0)\right\}$ is the standard basis in $\cel^n$.
Denote by $\Delta$ (resp. $\Delta(S(\Lambda))$) the corresponding
simplex (resp. \emph{Newton convex body}) in $\re^n$ (note that
$\Delta\subseteq\Delta(S(\Lambda))$). Then from \cite[Theorem
11.2]{askold} we obtain that $1 = \vol(\Delta(S(\Lambda))) \geq
\vol(\Delta)= 1$. So the Newton convex body of the rational map
$\Lambda$ is the standard simplex in $\re^n$. It would be
interesting to study the class of algebraic varieties $X$ for
which the latter property is satisfied for \emph{any} birational
map $X \dashrightarrow X$.
\end{remark}

\bigskip

\section{Intermedia: one group-geometric argument}
\label{section:intermedia-group}

\refstepcounter{equation}
\subsection{Two sets of generators in $\widetilde{\li_n}$}
\label{subsection:gens}

Let $G_{ne}$ be the set of all $f \in \widetilde{\li_n}$ such that
$a(f) \not\sim f$ for every $a\in\f_2\setminus{\{1\}}$ (cf.
Lemma~\ref{theorem:tilde-g-aut}). Similarly, let $G_e$ be the set
of all $f \in \widetilde{\li_n}$ such that $a(f) \sim f$ for all
$a\in\f_2$. In general, for any $g\in\widetilde{\li_n}$, let
$E_g\subseteq\f_2$ be the group of those $a$ for which $a(g)\sim
g$ (i.\,e. $E_f=\f_2$ for all $f\in G_e$).

\begin{example}
\label{example:g-n-e} It is easy to see that both sets $\dia_n
\cap G_{ne}$ and $\dia_n \cap G_{e}$ are infinite. Note also that
$G_{ne}$ and $G_e$ are stable under the conjugation and inversion
in $\widetilde{\li_n}$.
\end{example}

\refstepcounter{equation}
\subsection{The tree $\mathcal{T}$}
\label{subsection:the-tree}

Recall that $\f_2$ acts freely, transitively and isometrically on
a (four\,-\,valent) tree $\mathcal{T}$ (cf. Figure \ref{fig-1}
below). Furthermore, if $\mathcal{X}$ is a Riemann surface of
genus $2$, the group $\f_2$ appears in the \emph{Schottky
uniformization} of $\mathcal{X}$ (see e.\,g. \cite{manin}).
Namely, since $\f_2\subset PGL_2(\com)$, one obtains a natural
$\f_2$\,-\,action on $\p^1(\com)$. Let $S\subset\p^1(\com)$ be the
closure of the set of attractive and repulsive fixed points for
all $\gamma\in\f_2$. The complement $\Omega := \p^1(\com)\setminus
S$ is connected, $\Omega = \displaystyle\bigcup_{\gamma\in\f_2}
\gamma\cdot D$ for $D$ being the exterior domain of four
non\,-\,intersecting circles on the Riemann sphere $\p^1(\com)$,
and $\mathcal{X} = \f_2\backslash\Omega$ for the proper
discontinuous action $\f_2\circlearrowright\Omega$.

This amounts to the next

\begin{lemma}
\label{theorem:q-i-tree-omega}
$\mathcal{T}\sim_{\text{q.-i.}}\Omega$ (the latter being a domain
in $\p^1(\com)$).
\end{lemma}

Further, given $f \in G_e$ let us suppose for a moment that
$a(f^c) \ne f^c$ for all $a \in \f_2$, $c\in\widetilde{\li_n}$.
Identify $\mathcal{T}$ with its set of vertices
$\left\{a(f)\right\}_{a \in \f_2}$ and similarly introduce the
tree $\mathcal{T}^f_c:=\left\{a(f^c)\right\}_{a \in \f_2}$ (thus
$\mathcal{T}^f_c$ is another copy of $\mathcal{T} =
\mathcal{T}^f_1$). Then $\mathcal{T}^f_c$ carries a metric
$\di(*,*)$, coming from the word metric on $\f_2$, so that
$\di(a(f^c), b(f^c)) := \di(ab^{-1},1)$ for all
$a,b\in\f_2$.\footnote{~Note that $f^c\in G_e$ and $E_{f^c} = E_f$
for all $c$ (cf. Example~\ref{example:g-n-e}).} Now, gluing
$\mathcal{T}^f_c$ with $\mathcal{T}^f_{a(c)a}$ (isometrically) via
$b(f^c) \mapsto b^a(f^{a(c)a})$ for all $a, b\in \f_2$ (we regard
$a$, hence $a(c)a$, as an element in $\widetilde{\li_n}$ acting on
$f$ by conjugation), we may identify the metric space $$
\mathcal{C}_f :=
\displaystyle\bigsqcup_{c\in\widetilde{\li_n}}\mathcal{T}^f_c\slash_{\displaystyle\asymp}$$
with $\mathcal{T}$ (as sets). Here $\asymp$ is the equivalence
relation such that $b(f^c) \asymp b^a(f^{a(c)a})$ for all
$a,b,c$.\footnote{~In fact, we have $a'a(f) = a'(f^a) =
f^{a'(a)a'} = f^{a'a}$ for all $a',a\in \f_2$, which implies that
$\asymp$ is symmetric and transitive.} (Note also that the
assertion of Lemma~\ref{theorem:q-i-tree-omega} obviously holds
for $\mathcal{C}_f$ in place of $\mathcal{T}$.)

\begin{lemma}
\label{theorem:c-f-well-def} $\mathcal{C}_f$ is defined for
\emph{any} $f\in G_e$. More precisely, this $\mathcal{C}_f$ is
q.\,-\,i. to $\Omega$ and coincides with $T$ set-theoretically,
similarly as above.
\end{lemma}

\begin{proof}
We use the notation from Section~\ref{section:1-0-n}. Put ${\bf x}
:= [X_0\,:\,\cdots\,:\,X_n]$ and fix an arbitrary
$a\in\f_2\setminus{\{1\}}$. Then, since there is no
$c\in\widetilde{\li_n}$ such that $c^{-1}(g(c({\bf x}))) =
a^{-1}(g(a({\bf x})))$ for all $g\in \widetilde{\li_n}$ (because
$a\in\mathrm{Out}\,\widetilde{\li_n}$),\footnote{~Recall that
$a,g,c\in G$ act on $\p^n$. Then $c^{-1}(g(c({\bf x}))) =
a^{-1}(g(a({\bf x})))$ is understood as an identity between the
elements in $G$. In particular, this does not depend on the choice
of ${\bf x}$, so that the forthcoming $\{a;c\}$ is correctly
defined.} we can associate with $a(f^c)$ an \emph{ordered} pair
$\{a;c\}$. Now, since $\{a;c\}$ are all distinct for different
$a,c$, we repeat the previous construction of the trees
$\mathcal{T}^f_c$ (with $a(f^c)$ replaced by $\{a;c\}$). Finally,
we use the fact that $f\in G_e$ to glue the trees
$\mathcal{T}^f_c$ and $\mathcal{T}^f_{a(c)a}$ via $\asymp$ as
earlier, which gives $\mathcal{C}_f$ ($\sim_{\text{q.-i.}}\Omega$)
as wanted.
\end{proof}

\begin{remark}
\label{remark:warn} We should stress that the proof of
Lemma~\ref{theorem:c-f-well-def} really uses the specifics of
situation in order to define $\{a;c\}$ correctly. In general, for
the group $PGL_{n+1}(\com)$, say, and the automorphisms
$\f_2\subseteq\text{Aut}(\com)\subseteq\text{Out}$ acting on it,
or for the fundamental group $\pi_1(C,x)$ of a Riemann surface
$C\ni x$ acted by a free group of Dehn twists, the same
constructions of $G_{ne},G_e$, etc. carry on, but similar to the
above property (definition) of $\{a;c\}$ breaks because it
depends, non-trivially, on the choice of a basis for
$PGL_{n+1}(\com)$ (resp. a base point $x$ for $\pi_1(C,x)$).
\end{remark}

\begin{remark}
\label{remark:c-f} To say it in words, every $\mathcal{T}^f_c$ in
the definition of $\mathcal{C}_f$ corresponds to a ``coloring" of
$\mathcal{T}$ (one for each $c\in\widetilde{\li_n}$), compatible
with the $\f_2$\,-\,action (cf. Figure \ref{fig-1}). In turn, the
pairs $\{~;~\}$ from the proof of Lemma~\ref{theorem:c-f-well-def}
can be considered as ``\,local coordinates\," (with
$\mathcal{T}^f_c$ being ``\,local charts\,") on $\mathcal{T}$,
where the $\widetilde{\li_n}$\,-\,part corresponds to
``\,coordinate bases\,", while the $\f_2$\,-\,part is the
``\,coordinate values\,".
\end{remark}

In view of Lemma~\ref{theorem:c-f-well-def}, we will not
distinguish between $\mathcal{T}$ and $\mathcal{C}_f$ in what
follows, so that the tree $\mathcal{T}$ comes enhanced with an
additional structure (cf. Remark~\ref{remark:c-f}).

The next result may be considered as the glimpse of a certain
``\,Anosov property\," enjoyed by the elements from $G_e$, for one
may observe an analogy between the (hyperbolic) $\cel$\,-\,action
on $\text{Diff}$ (see the discussion in \cite[\S 2]{ambra-gromov}
or in \cite[\S 5]{gromov-hyp-manifolds} for instance) and the
$\f_2$\,-\,action on $G_e$ in our case, with assertions ``\,two
elements $f,g\in\text{Diff}$ are homotopic, $C^r$\,-\,close,
etc.\," being replaced by ``\,$f\sim g,~f,g \in G_e$\,".

\begin{proposition}
\label{theorem:m-2-good-generator} For every $f,g\in G_e$, the
group $E_{fg}$ is non\,-\,cyclic.\footnote{~The arguments below
work for the product of any $f_1,\ldots,f_m\in G_e$ and $m\geq
2$.}
\end{proposition}

\begin{proof}
Suppose that $E_{fg} = \left<b\right>$ for some $b\in\f_2$. Let us
glue $\mathcal{C}_f$ with $\mathcal{C}_g$ as follows:
$$
a(f^c) \asymp a(g^c)\ \text{for all}\ a\in
\f_2,\,c\in\widetilde{\li_n}.
$$
(Obviously, the latter $\asymp$ is compatible with the equivalence
relation used to construct $\mathcal{C}_f$ and $\mathcal{C}_g$
above, and so we keep the same symbol for both.) Again, since
$\f_2\subseteq\mathrm{Out}\,\widetilde{\li_n}$ and $f,g \in G_e$,
this construction is compatible with the $\f_2$\,-\,action. In
particular (to simplify the notation), we will assume that $a(f^c)
\ne f^c$, $a(g^c) \ne g^c$ for all $a \in \f_2$,
$c\in\widetilde{\li_n}$, as in Figure \ref{fig-1} below.

Further, in the preceding definition of the trees
$\mathcal{T}^f_c$ we can formally replace each $a(f^c)$ by the
conjugacy class $C_{a(f^cg^{c'})}$, with arbitrary $a\in\f_2$,
$c,c'\in\widetilde{\li_n}$, where again $C_{a(f^cg^{c'})}$ is
regarded as a (``\,\{value; coordinate\}\,") triple $\{a;c,c'\}$,
analogous to that in the proof of
Lemma~\ref{theorem:c-f-well-def}. Then we repeat the isometric
gluings of $\mathcal{T}^f_c$ with $\mathcal{T}^f_{a(c)a}$ (resp.
of $\mathcal{T}^g_{c'}$ with $\mathcal{T}^g_{a(c')a}$) to get the
tree $\mathcal{T}'\sim_{\text{q.-i.}}\mathcal{T}$ (the former
being (formally) identified with $\{C_{a(fg)}\}_{a\in\f_2}$),
carrying a free, transitive and isometric $\f_2$\,-\,action. Thus
the construction/enhancement of $\mathcal{T}'$ is essentially the
repetition verbatim of that for $\mathcal{T}$.

\begin{lemma}
\label{theorem:map-phi-t-prime} There is an $\f_2$\,-\,equivariant
(\emph{continuous}) map $\varphi: \mathcal{T} \longrightarrow
\mathcal{T}'$ of metric spaces which coincides with the quotient
map $\f_2 \to \f_2/\left<b\right>$ on the sets of vertices. In
particular, $\varphi$ is surjective.
\end{lemma}

\begin{proof}
Identify $\mathcal{T}$ with its chart
$\mathcal{T}^f_1\asymp\mathcal{T}^g_1$ and set $\varphi$ to be as
follows:
\begin{equation}
\label{eq-varphi-act} a(f) \asymp a(g) \mapsto
C_{a(fg)},~a\in\f_2.
\end{equation}
(Here we use the chart $\{C_{a(fg)}\}_{a\in\f_2}$ for
$\mathcal{T}'$ as well.) Clearly, the (set\,-\,theoretic) map
$\varphi$ is surjective and coincides with $\f_2 \to
\f_2/\left<b\right>$ on the sets of vertices, since $E_{fg} =
\left<b\right>$.

In order to extend $\varphi$ to a \emph{metric morphism} it
suffices to show the definition of $\varphi$ does not depend on
the ($\f_2$\,-\,equivariant) identification of $\mathcal{T}$ with
$\mathcal{T}^f_1$ (aka $\mathcal{T}^g_1$). This will follow once
we check the definition of $\varphi$ does not depend on replacing
$a$ by $a'a$ for an arbitrary fixed $a'\in\f_2$ and all $a$. But
that is why we need the \emph{enhanced} $\mathcal{T}$ and
$\mathcal{T}'$. Namely, regarding (as usual) $a'$ as an element in
$\widetilde{\li_n}$ acting on $f,g$ by conjugation, we simply pass
to the chart $\mathcal{T}^f_{a'}\asymp\mathcal{T}^g_{a'}$ of
$\mathcal{T}$ so that $\varphi$ now acts like this:
$$
a'aa'^{-1}(f^{a'}) \asymp a'aa'^{-1}(g^{a'}) \mapsto
C_{a'aa'^{-1}(f^{a'}g^{a'})}.
$$
(Here $\mathcal{T}'$ is also considered in its other chart
$\{C_{a(f^cg^{c'})}\}_{a\in\f_2}$.) Thus we have replaced
$\mathcal{T},\mathcal{T}'$ by their isometric copies and defined
$\varphi$ for these also (compatibly with \eqref{eq-varphi-act}).
This is correct because the charts
$\mathcal{T}^f_1\asymp\mathcal{T}^g_1$,
$\mathcal{T}^f_{a'}\asymp\mathcal{T}^g_{a'}$, etc. are (formally)
distinct by construction.
\end{proof}

Lemmas~\ref{theorem:q-i-tree-omega} and
\ref{theorem:map-phi-t-prime} yield a $1$\,-\,cycle fibration
$\Omega \longrightarrow\Omega$ which is q.\,-\,i. to $\varphi$:
\begin{figure}[h]
\includegraphics[scale=0.9]{crem.1}
\caption{~}\label{fig-1}
\end{figure}


But the latter is impossible for the domains in $\p^1(\com)$.
Proposition~\ref{theorem:m-2-good-generator} is proved.
\end{proof}
The set $G_e$ generates a normal subgroup $N$ in
$\widetilde{\li_n}$ with $N \cap \left<\slo_n\right> \supseteq
G_e\cap\dia_n \ne \{1\}$ (cf.
Section~\ref{section:theorem-group-e-n} and
Example~\ref{example:g-n-e}). We will see in
Section~\ref{section:proof-of-theorem} that the complement
$\widetilde{\li_n} \setminus G_{ne}\sqcup G_e$ contains an element
$\Lambda$ such that the group $E_{\Lambda}$ is \emph{cyclic}. This
together with Proposition~\ref{theorem:m-2-good-generator} implies
that $N \vartriangleleft\widetilde{\li_n}$. We will show that in
fact $\widetilde{\li_n}\setminus N \ni \Lambda$ for some
$\Lambda\in\left<\slo_n\right>\setminus\kr(\xi)$ (cf.
Proposition~\ref{theorem:sub-group-surj}), which easily yields
$\xi(N\cap\left<\slo_n\right>)\vartriangleleft\au_{n - 1}$ (see
the discussion after Corollary~\ref{theorem:cor-res-d} below),
hence Theorem~\ref{theorem:main}. (Notice by the way that
$N\cap\left<\slo_n\right>\not\subset \kr(\xi)$ because
$G_e\cap\dia_n\not\subset \kr(\xi)$.)

\begin{remark}
\label{remark:g-e-g-ne} It would be interesting to test
non\,-\,simplicity of any group $G$ satisfying $\f_2 \subseteq
\mathrm{Out}\,G$ and $G_e\cdot G_e\subseteq G_{ne}\sqcup G_e$ (cf.
Question~\ref{theorem:stupid} below).
\end{remark}

\bigskip

\section{Proof of Theorem~\ref{theorem:main}}
\label{section:proof-of-theorem}

\refstepcounter{equation}
\subsection{}
\label{subsection:free-aut}

We keep up with the previous notation. Let us assume in addition
that $n \geq 4$. Consider
$\Lambda\in\slo_n\subset\widetilde{\li_n}$ defined as follows:
\begin{eqnarray}
\label{eq-of-lam} X_i \mapsto X_0^{d - 1}X_i,\\
\nonumber X_4 \mapsto X_0^{d-1}X_4 + \Lambda_{d}(X_2, X_4, \ldots,
X_n)
\end{eqnarray}
for all $i \ne 0,1,4$ (cf. \eqref{eq-1-1-0-m}). Let us also
consider $a_1, a_2 \in SL_n'(\cel)$ defined by
$$
a_1:\, [1\,:\,X_1\,:\,\cdots\,:\,X_n] \mapsto
[1\,:\,X_1\,:\,X_2\,:\,X^{I_1}\,:\,X_4\,:\,\cdots\,:\,X_n],
$$
$$
a_2:\, [1\,:\,X_1\,:\,\cdots\,:\,X_n] \mapsto
[1\,:\,X_1\,:\,X^{I_2}\,:\,X_3\,:\,\cdots\,:\,X_n]
$$
for $I_1 := (1, -1, 1, 0, \ldots, 0)$ and $I_2 := (-1, 1, 1, 0,
\ldots, 0)$ (cf. Example~\ref{example:ex-1-1-0-n}).

\begin{lemma}
\label{theorem:a-1-a-2-f-2} The elements $a^2_1, a^2_2$ generate a
free subgroup $\f_2\subseteq SL_n'(\cel)$.
\end{lemma}

\begin{proof}
Take a matrix $\bigstar := \begin{pmatrix} 1 & 0 & 0
\\ 0 & 1 & 0 \\ 0 & a & b
\end{pmatrix}$ with some $a,b\in\com$. Then the matrix $X' := \bigstar a_1$ has entries $X'_{2,3}
= a,X'_{3,3} = -a + b$. Similarly, $X'' := \bigstar a_2$ has
$X''_{2,3} = a + b,X''_{3,3} = b$. Letting $a := 1$, $b := 0$ and
$a := 0$, $b := 1$ we obtain a homomorphism from the group
generated by $a^2_1$, $a^2_2$ onto the subgroup $\Gamma \subseteq
SL_2(\cel)$ generated by the matrices $\begin{pmatrix} 1 & 0 \\
-2 & 1\end{pmatrix}$ and $\begin{pmatrix} 1 & 2
\\ 0 & 1\end{pmatrix}$. Hence it suffices to show that $\Gamma
\simeq\f_2$. But the latter follows from the Ping\,-\,Pong Lemma.
\end{proof}

Notice that $a_1\Lambda a_1^{-1} = \Lambda$ (i.\,e. $a_1 \in
E_{\Lambda}$ in the notation of {\ref{subsection:gens}}). On the
other hand, we have the following:

\begin{proposition}
\label{theorem:not-invariant} There exists $\Lambda$ of the form
\eqref{eq-of-lam} such that $a\Lambda a^{-1} \not\sim \Lambda$ for
any $a \in \f_2\setminus{\{a_1^k\}}_{k\in\mathbb{Z}}$.
\end{proposition}

\begin{proof}
Suppose that $a\Lambda a^{-1} = g\Lambda g^{-1}$ for some
$a\in\f_2\setminus{\{a_1^k\}}_{k\in\mathbb{Z}}$ and
$g\in\widetilde{\li_n}$. We will exclude only the case $a := a_2$
for the general case is treated similarly. The map $a_2\Lambda
a_2^{-1}$ acts as follows:
\begin{eqnarray}
\nonumber X_i \mapsto X_3^dX_0^{d - 1}X_i,\\
\nonumber X_4 \mapsto X_0^{d-1}X_3^dX_4 + \Lambda_{d}(X_1X_2,
X_4X_3, \ldots, X_nX_3)
\end{eqnarray}
for all $i \ne 0,1,4$. At this stage we assume that $\Lambda_d \ne
0 \mod (X_4, \ldots, X_n)$. Then the map $a_2\Lambda a_2^{-1}$
contracts the hyperplane $H := (X_3 = 0)$ to a point. On the other
hand, we have

\begin{lemma}
\label{theorem:not-invariant-1} $g\Lambda g^{-1}$ does not
contract $H$.
\end{lemma}

\begin{proof}
Write $g = g_1 \cdot g_2$ as in the proof of
Lemma~\ref{theorem:tilde-g-aut}. It suffices to consider only the
case when $g_1 = 1$ (for $g_1 \in PGL_n(\com)$ a priori).

Put $g^{-1}_i/g^{-1}_0 := X_i + \varepsilon$ for the components of
$g^{-1}$, $1\leq i\leq n$, and some (varying) $0\leq \varepsilon
\ll 1$. Let also $O := [0\,:\,1\,:\,0\,:\,\cdots\,:\,0]$.

Now, since $g = g_2 \in \li_n$ and $\Lambda(O) = O$, one can
easily see that $g\Lambda g^{-1}$ asymptotically equals $\Lambda$.
More precisely, as in the proof of Lemma~\ref{theorem:cruc-lm-1}
one finds that $(\Lambda g^{-1})_i = \Lambda_i + \varepsilon$ for
the components of $\Lambda g^{-1}$ and $\Lambda$, respectively.
Furthermore, since $\varepsilon$ can be expressed as an analytic
function in $\Lambda_i$, locally near the point $O$, by the same
argument (with $X_i$ and $\Lambda_i$ interchanged) we get
$(g\Lambda g^{-1})_i = \Lambda_i + \varepsilon$ for the components
of $g(\Lambda g^{-1})$. Thus $g\Lambda g^{-1}$ can not contract
$H$.
\end{proof}

From Lemma~\ref{theorem:not-invariant-1} we obtain a contradiction
$a_2\Lambda a_2^{-1} \ne g\Lambda g^{-1}$. This finishes the proof
of Proposition~\ref{theorem:not-invariant}.
\end{proof}

Let $\li^{\Pi}_n\subseteq\widetilde{\li_n}$ be the maximal
subgroup preserving the hyperplane $\Pi = (X_1 = 0)$. Take
$\Lambda$ as in the proof of
Proposition~\ref{theorem:not-invariant} and let $\Lambda_0$ be the
restriction of $\Lambda$ to $\Pi$. Then from (the proof of)
Proposition~\ref{theorem:not-invariant} we get the following (with
respect to the induced $\rho(G)$\,-\,action on $\Pi$):

\begin{corollary}
\label{theorem:cor-res-d} $a_1\Lambda_0 a_1^{-1} = \Lambda_0$ and
$a\Lambda_0 a^{-1} \not\sim \Lambda_0$ in
$\li^{\Pi}_n\big\vert_{\Pi}$ for every $a \in
\f_2\setminus{\{a_1^k\}}_{k\in\mathbb{Z}}$.
\end{corollary}

We have $\slo_n\subseteq\li^{\Pi}_n$ and
$\f_2\subseteq\mathrm{Out}\,\li^{\Pi}_n$ via the induced
$\rho(G)$\,-\,action on $\Pi$ (cf.
Lemmas~\ref{theorem:a-1-a-2-f-2} and \ref{theorem:tilde-g-aut}).
Then the arguments of Section~\ref{section:intermedia-group}, with
the extra condition ``\,modulo $X_1$\," added, apply verbatim to
show that $(N\cap\li^{\Pi}_n)\big\vert_{\Pi}$ is a proper normal
subgroup of $\li^{\Pi}_n\big\vert_{\Pi}$ such that
$(N\cap\left<\slo_n\right>)\big\vert_{\Pi}\ne\{1\}$ and
$\Lambda_0\not\in N\big\vert_{\Pi}$ (for the latter we have also
used Corollary~\ref{theorem:cor-res-d}).

\begin{lemma}
\label{theorem:final-lemma}
$\xi(N\cap\left<\slo_n\right>)\vartriangleleft\au_{n - 1}$.
\end{lemma}

\begin{proof}
Indeed, we have

\begin{itemize}

    \item $\left<\slo_n\right>\big\vert_{\Pi} := \xi(\left<\slo_n\right>) = \au_{n-1}$ (see Proposition~\ref{theorem:sub-group-surj}),

    \item $\Lambda\in\slo_n$ and $\Lambda_0 =
\xi(\Lambda)\not\in\xi(N\cap\left<\slo_n\right>)=(N\cap\left<\slo_n\right>)\big\vert_{\Pi}$
because $\Lambda_0\not\in N\big\vert_{\Pi}$,

    \item $N\cap\left<\slo_n\right>\not\subset
\kr(\xi)$.

\end{itemize}

This shows that $\xi(N\cap\left<\slo_n\right>)\ne\{1\}$,
$\au_{n-1}$, i.\,e.
$\xi(N\cap\left<\slo_n\right>)\vartriangleleft\au_{n - 1}$.
\end{proof}

Lemma~\ref{theorem:final-lemma} finishes the proof of
Theorem~\ref{theorem:main}.

\bigskip

\section{Final comments}
\label{section:theorem-proofs}

The way we used the groups $G$, $\li_n$, $\slo_n$, etc. to prove
Theorem~\ref{theorem:main} makes it reasonable to develop the
preceding arguments more systematically and study other subgroups
in $\cg_n$ which ``\,behave expectedly at infinity\,". Let us
advocate this thesis by proving the following:

\begin{proposition}[{cf. \cite[{\bf 5.1}]{cerveau-deserti}}]
\label{theorem:g-is-not-embeddable} For any, not necessarily
algebraically closed field ${\bf k}\subset\com$, the group $\cg_n$
is not embedable into $GL_m(\com)$ for all $m \in
\na\cup\{\infty\}$ and $n \geq 2$.\footnote{~After the text has
been written, I was informed by S.\,Cantat about
http://perso.univ-rennes1.fr/serge.cantat/Articles/cnl-5.jpg,
where a similar statement had been proved (via a
group\,-\,theoretic argument) for every \emph{finite} $m \in
\na$.}
\end{proposition}

\begin{proof}
Take $g_1\in\dia_n$ and $g_2\in SL'_n(\cel)$ any unipotent
element. Consider the group $E := \left<g_1, g_2\right>\subset G$.
Let us also suppose that $g_1^2 = 1$. We can always choose $g_1,
g_2$ in such a way that $E = \left<g_1\right>^{\oplus k} \rtimes
\left<g_2\right>$ for some $2\leq k \leq n$.

\begin{lemma}
\label{theorem:g-is-not-embeddable} $E$ is not embeddable into
$GL_m(\com)$ for any $m \in \na\cup\{\infty\}$.
\end{lemma}

\begin{proof}
Here we follow the paper \cite{gromov-2}. Suppose that $E \subset
GL_m(\com)$ for some $m$.

Consider a word metric $\di_{E}$ on $E$ and the corresponding
metric space $(E,\di_{E})$. Then, since the extension
$\left<g_1\right>^{\oplus k} \rtimes \left<g_2\right>\subset
GL_m(\com)$ is non\,-\,trivial, we may assume that $2 \leq m <
\infty$, which gives a natural isometric embedding
$\big(E,\di_{E}) \hookrightarrow (\left<\left<g_1\right>^{\oplus
k}, S^1\right> \rtimes\re,\di\big)$, for
$\log\big(\di\big\vert_{E}\big) = \di_{E}$, such that
$(\left<\left<g_1\right>^{\oplus k}, S^1\right>
\rtimes\re,\di\big)$ is a hyperbolic space and $s \times \re$ is a
horocycle for all $s \in \left<\left<g_1\right>^{\oplus k},
S^1\right>$ (see \cite[{\bf 2.B},\,(d),\,(e),\,(f)]{gromov-2}). In
particular, since $g_1^2 = 1$ and $\left<g_2\right> = \cel\subset
\re$, this implies that $\text{Con}_{\infty}(E)$, the
\emph{asymptotic cone} of $E$ (with induced metric), is totally
disconnected (\emph{op. cit.}).

On the other hand, since $g_2^a \circ g \circ g_2^b = g' \circ
g_2^{c}$ for all $a, b \in\cel$, $g\in\left<g_1\right>^{\oplus k}$
and some $c := c(a, b) \in \cel$, $g' :=
g'(g)\in\left<g_1\right>^{\oplus k}$, the group $E\subset\li_n$
(obviously) acts as $\cel = \left<g_2\right>$ on the Berkovich
spectrum of $\com^n$ (cf. \cite[Section 5]{max-yura}). In
particular, we obtain that $(E,\di_E)$ is q.\,-\,i. to $\cel$ with
the corresponding word metric (see \cite[{\bf 0.2.C}]{gromov-2}),
which implies that $\text{Con}_{\infty}(E) = \re$ with the usual
metric (see \cite[{\bf 2.B}, (a)]{gromov-2}). This contradicts the
previous paragraph.
\end{proof}

Lemma~\ref{theorem:g-is-not-embeddable} proves
Proposition~\ref{theorem:g-is-not-embeddable}.
\end{proof}

\begin{remark}
\label{remark:comment-2} It would be interesting to construct
examples of algebraic varieties $X$ over a number field $F$ for
which the above non\,-\,embeddability result for $\cg_n$ provides
a non\,-\,trivial obstruction to rationality of $X$ over $F$.
\end{remark}

Finally, Corollary~\ref{theorem:rho-hom} relates $\li_n$ to
hyperbolic groups and groups with small cancellation (cf.
\cite{cou}), which together with results of
Section~\ref{section:theorem-group-e-n} makes one ask the next

\begin{question}
\label{theorem:stupid} Let $G$ be a group such that $\f_2
\subseteq \mathrm{Out}\,G$. Is $G$ non\,-\,simple?
\end{question}

Unfortunately, the answer to Question~\ref{theorem:stupid} is
negative in general, as the case of the group $G :=
PGL_{n+1}(\com)$ (with $\f_2\subset
\mathrm{Gal}(\com/\ra)\subseteq \mathrm{Out}\,G$)
shows.\footnote{~Pointed out by M.\,Gromov.} However, the latter
indicates an interesting difference between the groups
$\widetilde{\li_n}$ and $PGL_{n+1}(\com)$, which together with the
proof of Theorem~\ref{theorem:main} suggests a way to attack the
(non\,-)\,simplicity of $\cg_n$ for all $n \geq 2$ (basically, one
constructs a normal subgroup $N\subseteq\cg_n$\footnote{~Provided
that $\f_2\subseteq \mathrm{Out}\,\cg_n$. For example, one may
attain this for $\mathrm{Gal}(\com/\ra)\subseteq
\mathrm{Out}(\cg_n)$ (cf. \cite{cerveau-deserti}), but then the
construction of $N$ may be no longer valid.} exactly as in
Section~\ref{section:intermedia-group} above (cf.
Remark~\ref{remark:g-e-g-ne}), and tries to show that $N \ne
\cg_n$, arguing as in Section~\ref{section:proof-of-theorem} for
instance).

\bigskip

\thanks{{\bf Acknowledgments.} The work owes much to the Hurricane Irene (August, 2011), due to which
I was stuck in Boston for couple of extra days. I am grateful to
A.\,Postnikov for hospitality. Also thanks to F.\,Bogomolov,
M.\,Gizatullin, and A.\,Khovanskii for their patience and
encouragement. Partial financial support was provided by CRM
fellowship and NSERC grant.


\bigskip

\begin{thebibliography}{20}

\bibitem{bergman}
G. M. Bergman, The logarithmic limit-set of an algebraic variety,
Trans. Amer. Math. Soc. {\bf 157} (1971), 459~--~469.

\smallskip

\bibitem{can-lam}
S. Cantat\ and\ S. Lamy, Normal subgroups in the Cremona group,
Acta Math. {\bf 210} (2013), no.~1, 31~--~94.

\smallskip

\bibitem{cou}
R. Coulon, Automorphismes ext\'erieurs du groupe de Burnside
libre, Institut de Recherche Math\'ematique Avanc\'ee,
Universit\'e de Strasbourg, Strasbourg, (2010).

\smallskip

\bibitem{ambra-gromov}
G. D'Ambra\ and\ M. Gromov, Lectures on transformation groups:
geometry and dynamics, in {\it Surveys in differential geometry
(Cambridge, MA, 1990)}, 19~--~111, Lehigh Univ., Bethlehem, PA.

\smallskip

\bibitem{danilov}
V. I. Danilov, Non-simplicity of the group of unimodular
automorphisms of an affine plane, Mat. Zametki {\bf 15} (1974),
289~--~293.

\smallskip

\bibitem{dan}
V. I. Danilov, Polyhedra of schemes and algebraic varieties, Math.
USSR-Sb. {\bf 26} (1975), no.~1, 137~--~149.

\smallskip

\bibitem{cerveau-deserti}
J. D\'eserti, Sur les automorphismes du groupe de Cremona, Compos.
Math. {\bf 142} (2006), no.~6, 1459~--~1478.

\smallskip

\bibitem{kapranov}
M. Einsiedler, M. Kapranov,\ and\ D. Lind, Non-Archimedean amoebas
and tropical varieties, J. Reine Angew. Math. {\bf 601} (2006),
139~--~157.

\smallskip

\bibitem{gromov-2}
M. Gromov, Asymptotic invariants of infinite groups, in {\it
Geometric group theory, Vol.\ 2 (Sussex, 1991)}, 1~--~295, London
Math. Soc. Lecture Note Ser., 182 Cambridge Univ. Press,
Cambridge.

\smallskip

\bibitem{gromov-pol-growth}
M. Gromov, Groups of polynomial growth and expanding maps, Inst.
Hautes \'Etudes Sci. Publ. Math. No. 53 (1981), 53~--~73.

\smallskip

\bibitem{gromov-hyp-manifolds}
M. Gromov, Hyperbolic manifolds, groups and actions, in {\it
Riemann surfaces and related topics: Proceedings of the 1978 Stony
Brook Conference (State Univ. New York, Stony Brook, N.Y., 1978)},
183~--~213, Ann. of Math. Stud., 97 Princeton Univ. Press,
Princeton, NJ.

\smallskip

\bibitem{askold}
K. Kaveh\ and\ A. Khovanskii, Algebraic equations and convex
bodies, in {\it Perspectives in analysis, geometry, and topology},
263~--~282, Progr. Math., 296 Birkh\"auser/Springer, New York.

\smallskip

\bibitem{max-yura}
M. Kontsevich\ and\ Yu. Tschinkel, Non-archimedean K\"ahler
geometry, Preprint (2011).

\smallskip

\bibitem{manin}
Yu.\ I. Manin, Three-dimensional hyperbolic geometry as
$\infty$-adic Arakelov geometry, Invent. Math. {\bf 104} (1991),
no.~2, 223~--~243.

\smallskip

\bibitem{grisha}
G. Mikhalkin, Real algebraic curves, the moment map and amoebas,
Ann. of Math. (2) {\bf 151} (2000), no.~1, 309~--~326.

\smallskip

\bibitem{andrei-1}
A. Okounkov, Brunn-Minkowski inequality for multiplicities,
Invent. Math. {\bf 125} (1996), no.~3, 405~--~411.

\smallskip

\bibitem{andrei-2}
A. Okounkov, Why would multiplicities be log-concave?, in {\it The
orbit method in geometry and physics (Marseille, 2000)},
329~--~347, Progr. Math., 213 Birkh\"auser, Boston, Boston, MA.

\smallskip

\bibitem{prok}
Y. Prokhorov, Simple finite subgroups of the Cremona group of rank
3, J. Algebraic Geom. {\bf 21} (2012), no.~3, 563~--~600.

\smallskip

\bibitem{serre-1}
J.-P. Serre, Le groupe de Cremona et ses sous-groupes finis,
Ast\'erisque No. 332 (2010), Exp. No. 1000, vii, 75~--~100.

\smallskip

\bibitem{shaf}
I. R. Shafarevich, On some infinite-dimensional groups. II, Izv.
Akad. Nauk SSSR Ser. Mat. {\bf 45} (1981), no.~1, 214~--~226.

\end{thebibliography}
\end{document}